\documentclass[11pt]{article}
\usepackage[english]{babel}
\usepackage{amssymb}
\usepackage{amsmath}
\usepackage{amsthm}
\usepackage[a4paper,margin=2.54cm]{geometry}
\usepackage{graphicx}
\usepackage{todonotes}
\usepackage{comment}
\usepackage{appendix}

\theoremstyle{plain}
\newtheorem*{theorem*}{Theorem}
\newtheorem{theorem}{Theorem}[section] 
\newtheorem{lemma}[theorem]{Lemma}
\newtheorem{proposition}[theorem]{Proposition}
\newtheorem{corollary}[theorem]{Corollary}
\newtheorem{conjecture}[theorem]{Conjecture}

\theoremstyle{definition}
\newtheorem{definition}[theorem]{Definition}
\newtheorem{remark}[theorem]{Remark}
\newtheorem{example}[theorem]{Example}

\makeatletter
\@addtoreset{equation}{section}
\makeatother
\numberwithin{equation}{section}

\DeclareMathOperator{\sgn}{sgn}

\usepackage{authblk}
\title{Equivalent symmetric kernels of determinantal point processes}

\author{Marco Stevens}
\affil{KU Leuven, Department of Mathematics, 
	
	Celestijnenlaan~200B box 2400, 3001 Leuven, Belgium. 
	
	E-mail:~{\tt marco.stevens@kuleuven.be}
}
\date{\today}                     
\usepackage{hyperref}
\hypersetup{colorlinks, citecolor=black, linkcolor=black}
\AtBeginDocument{} 

\begin{document}
\maketitle

\begin{abstract}
Determinantal point processes are point processes whose correlation functions are given by determinants of matrices. The entries of these matrices are given by one fixed function of two variables, which is called the kernel of the point process. It is well-known that there are different kernels that induce the same correlation functions. We classify all the possible transformations of a kernel that leave the induced correlation functions invariant, restricting to the case of symmetric kernels.
\end{abstract}

\section{Introduction and main result}
\label{sectionintroduction}

Point processes appear in a wide variety of mathematical subjects that model the random behaviour of a discrete set of points, such as random matrix theory. Determinantal point processes \cite{Borodin, Johansson, Soshnikov} exhibit a particularly convenient algebraic and analytic structure that allows a rich analysis. Such processes are characterized by the fact that their correlation functions are of a determinantal form. To be more precise, if we write $\rho_n$ for the $n^\mathrm{th}$ correlation function of a certain given point process on some measure space $X$, then there is a function $K: X^2 \rightarrow  \mathbb{F}$ such that for any $n>0$ and any tuple $(x_1,\dots,x_n)\in X^n$, we have
\[\rho_n(x_1,\dots,x_n) = \det(K(x_i,x_j))_{i,j=1}^n.\]
Here, $\mathbb{F}$ is a suitable field, and $K$ is called the \textbf{correlation kernel} of the point process. For most (if not all) practical purposes, the field $\mathbb{F}$ is either the set of real or complex numbers.

The question that we (partially) answer in this paper is the following: to what extent is the correlation kernel $K$ unique? In other words, given a certain function $K: X^2\rightarrow \mathbb{F}$, can we classify all the functions $Q: X^2 \rightarrow \mathbb{F}$ such that we have
\begin{equation}
\label{eq:equivalentkernels}
\det(K(x_i,x_j))_{i,j=1}^n = \det(Q(x_i,x_j))_{i,j=1}^n
\end{equation}
for all $n>0$ and all tuples $(x_1,\dots,x_n)\in X^n$? If $K$ and $Q$ are related as in \eqref{eq:equivalentkernels}, we call $K$ and $Q$ \textbf{equivalent kernels}. In answering this question, we regard the functions $K$ and $Q$ as mere functions, not as kernels of determinantal point processes, although this is the original interest. Furthermore, we neglect the measure space structure of $X$ and just consider $X$ as a set. For what follows, it is mostly unimportant which field $\mathbb{F}$ is; only the case where $\mathbb{F}$ has characteristic 2 needs more care. Unless otherwise specified, $\mathbb{F}$ is an arbitrary field. Before one is interested in making a classification, one should know some natural examples.

\begin{example}
\label{ex:transposition} 
If $K: X^2 \rightarrow \mathbb{F}$ is a function, define
\[Q(x,y)=K(y,x) \qquad x,y \in X.\]
Then equation \eqref{eq:equivalentkernels} is fulfilled because of the basic identity $\det(A^T)=\det(A)$ for any matrix $A$, where $A^T$ is the transpose of $A$. We refer to this transformation $K\mapsto Q$ as \textbf{transposition}.
\end{example}

\begin{example}
\label{ex:conjugation}
For a slightly more involved example, take a function $g: X \rightarrow \mathbb{F}\setminus \{0\}$ and define
\begin{equation}
\label{eq:gaugetransform}
Q(x,y) = g(x)g(y)^{-1} K(x,y), \qquad x,y \in X.
\end{equation}
Evaluating both sides of \eqref{eq:equivalentkernels} as sums over permutations, one sees that $K$ and $Q$ are equivalent. We call this transformation the \textbf{conjugation transformation} and $g$ the \textbf{conjugation function}.
\end{example}

\begin{remark}
\label{rem:cocycle}
The essential property of the conjugation transformation is the fact that the 2-variable function $c(x,y)=g(x)g(y)^{-1}$ is an example of a \textbf{cocycle}; that means that for all $r\geq 1$ and all tuples $(z_1,z_2,\dots,z_r)\in X^r$, one has that
\begin{equation}
\label{eq:cocyclecondition}
c(z_1,z_2)\cdot c(z_2,z_3) \cdots c(z_{r-1},z_r) \cdot c(z_r,z_1)=1.
\end{equation}
If $c$ is an arbitrary cocycle, then one can define
\begin{equation}
\label{eq:coycletransform}
Q(x,y)=c(x,y)K(x,y),
\end{equation}
and similarly as for the conjugation transformation conclude that $Q$ and $K$ are equivalent. However, one easily shows that for any cocycle $c$, there is a function $g: X\rightarrow \mathbb{F}\setminus \{0\}$ such that $c(x,y)=g(x)g(y)^{-1}$. Simply choose any point $x_0\in X$ and define $g(x)=c(x,x_0)$. Therefore, the a priori more general cocycle transformation \eqref{eq:coycletransform} does not yield anything new.
\end{remark}

The transposition and conjugation transformations are canonical; in fact, the conjecture \cite{Bufetov_personal} is the following.
\begin{conjecture}
If $K$ and $Q$ are equivalent kernels as in \eqref{eq:equivalentkernels}, then they can be transformed into one another by transposition and conjugation transformations.
\end{conjecture}

There is no known strategy to solve this conjecture in its full generality. However, if one considers primary examples of determinantal point processes, one observes that many of their kernels are in fact \textbf{symmetric}, which means that $K(x,y)=K(y,x)$ for all $x,y \in X$. This is true for the Christoffel-Darboux kernel associated to a sequence of orthogonal polynomials, but also for the Sine, Airy and Bessel kernel that arise as universal kernels for the asymptotic behaviour in random matrices and related subjects \cite{Borodin, Bufetov_rigidity, Deift_Kriecherbauer_McLaughlin_Venakides_Zhou,Johansson,Kuijlaars_2004, Soshnikov, TracyWidom}.

Therefore, we restrict our analysis to the case where both $K$ and $Q$ are symmetric. Of course, in the case that $K$ and $Q$ are symmetric, the matrices that appear in \eqref{eq:equivalentkernels} are symmetric as well, so the transposition transformation that is discussed in Example \ref{ex:transposition} now trivializes to the identity. It is also not possible to use every conjugation function without violating the symmetry: up to overall scalar multiplication, the only conjugation functions $g$ that one can take are those that take values in $\{\pm 1\}$. Our main result says precisely that the conjugation transformations are the only possible transformations that yield an equivalent kernel.

\begin{theorem}\label{thm:maintheorem}
Suppose that $X$ is a set, let $\mathbb{F}$ be a field and let $K,Q: X^2 \rightarrow \mathbb{F}$ be symmetric kernels. Then $K$ and $Q$ are equivalent (i.e., equation \eqref{eq:equivalentkernels} holds) if and only if there is a conjugation function $g: X \rightarrow \{\pm 1\}$ such that \eqref{eq:gaugetransform} holds.
\end{theorem}

We note that by the above one of the implications in this statement is clear; the existence of the conjugation function $g$ such that \eqref{eq:gaugetransform} holds implies that the kernels $K$ and $Q$ are equivalent. The main point of the proof of this theorem is hence to construct the conjugation function $g$ for any given pair of equivalent kernels $K$ and $Q$.

\begin{remark}
If $\mathbb{F}$ is a field of characteristic 2, then $1=-1$. In this case, the conjugation transformation also trivializes, and Theorem \ref{thm:maintheorem} actually states that two symmetric kernels $K$ and $Q$ are equivalent if and only if $K=Q$. Fields with characteristic 2 need special attention in what follows; see Corollary \ref{cor:char2} and in the proof of Proposition \ref{prop:loopcondition}.
\end{remark}

\begin{remark}
\label{rem:pm1andowninverse}
Note that the set $\{\pm 1\}$ is precisely the set of all elements that are their own inverse. We use this characterization of $\{\pm 1\}$ later on.
\end{remark}

\begin{remark}
\label{rem:uniquenessg}
The conjugation function $g$ that establishes \eqref{eq:gaugetransform} is not unique. Namely, if one would define $\tilde{g}(x)=-g(x)$, one immediately sees that $\tilde{g}$ is also a suitable conjugation function. In Remark \ref{rem:definiteuniqueness}, we give a complete classification of all the conjugation functions $g$ that yield the same conjugation transformation.
\end{remark}

\begin{remark}
\label{rem:resultintermsofactions}
Our main theorem can be formulated in terms of group actions. Namely, observe that the set 
\[G=\{g: X \rightarrow \{\pm 1\}\}\] 
is a group under pointwise multiplication. This group $G$ acts on the space
\[\operatorname{SymKer}(X)=\{K: X^2 \rightarrow \mathbb{F} \mid K \textrm{ is symmetric}\}\]
of symmetric kernels via the conjugation transformation \eqref{eq:gaugetransform}. The fact that conjugation transformations yield equivalent kernels then means that an orbit of this group action lies within one and the same equivalence class of the equivalence relation \eqref{eq:equivalentkernels}, i.e., we have the natural surjection
\[\operatorname{SymKer}(X) / G \twoheadrightarrow \operatorname{SymKer}(X) / \sim ,\]
where $\sim$ denotes the equivalence relation given by \eqref{eq:equivalentkernels}. Theorem \ref{thm:maintheorem} makes this stronger, by stating that this surjection is in fact an isomorphism, i.e., that the orbits of the action of $G$ are precisely the equivalence classes of \eqref{eq:equivalentkernels}.
\end{remark}

The rest of this article is concerned with the proof of Theorem \ref{thm:maintheorem} and the structure of the text reveals the steps in the proof. In Section \ref{sectionsetup} we study a pair of equivalent symmetric kernels $K, Q: X^2 \rightarrow \mathbb{F}$ and derive their first shared properties. Most importantly, we define a function $S$ that plays a similar role as the cocycle $c$ in \eqref{eq:coycletransform}, but is only defined on an appropriate subset of $X^2$. In order to deal with the fact that this function $S$ is not defined on all elements of $X^2$, we introduce (what we call) the \textit{equivalent kernel graph} in Section \ref{sectiongraph}. There, we also recall some notation and concepts from graph theory. Subsequently, in Section \ref{sectioncocyclecondition}, we prove the analogue of the defining condition \eqref{eq:cocyclecondition} of cocycles for the function $S$ in terms of the equivalent kernel graph. Finally, in Section \ref{sectiongauge}, we show that this condition implies the existence of a conjugation function $g$ such that \eqref{eq:gaugetransform} holds, and thereby we prove the remaining implication in the statement of Theorem \ref{thm:maintheorem}.

\section{The transition function}
\label{sectionsetup}

Throughout the rest of this article, we assume that $K$ and $Q$ are equivalent symmetric kernels, i.e., that \eqref{eq:equivalentkernels} holds. The goal of the rest of this article is to show that $K$ and $Q$ are in fact related by a conjugation transformation. This assumption implies the following fundamental lemma.
\begin{lemma}
For all $x,y\in X$, we have that
\begin{align}
K(x,x) &= Q(x,x), \label{eq:firstfundamental} \\
K(x,y)^2 &= Q(x,y)^2. \label{eq:secondfundamental}
\end{align}
\end{lemma}
\begin{proof}
For \eqref{eq:firstfundamental}, let $x\in X$ and specify \eqref{eq:equivalentkernels} to the case $n=1$. For \eqref{eq:secondfundamental}, we specify to the case $n=2$ and $(x,y)\in X^2$, to obtain
\[K(x,x)K(y,y)-K(x,y)K(y,x)=Q(x,x)Q(y,y)-Q(x,y)Q(y,x) \qquad x,y\in X.\]
Now use the symmetry of both $K$ and $Q$ and \eqref{eq:firstfundamental} to obtain that \eqref{eq:secondfundamental} holds.
\end{proof}

The relation \eqref{eq:secondfundamental} motivates the following definition.
\begin{definition}
The \textbf{zero set} of the two equivalent kernels $K$ and $Q$ is given by
\begin{equation}\label{eq:defzeroset}
Z=\{(x,y)\in X^2 \mid K(x,y)=Q(x,y)=0\}.
\end{equation}
Furthermore, the function $S: X^2 \setminus Z \rightarrow \mathbb{F}\setminus \{0\}$ is defined by
\begin{equation}\label{eq:defS}
S(x,y)=Q(x,y) K(x,y)^{-1},
\end{equation}
and is called the \textbf{transition function} from $K$ to $Q$.
\end{definition}

To prove Theorem \ref{thm:maintheorem}, we show that the transition function $S$ satisfies (most of) the properties of the cocycle $c$ as in Remark \ref{rem:cocycle}. The main issue for this is the fact that the transition function $S$ is only defined on $X^2 \setminus Z$. Hence the cocycle condition \eqref{eq:cocyclecondition} that holds for any tuple $(x_1,x_2,\dots,x_n)\in X^n$ should be replaced by only requiring that
\begin{equation}
\label{eq:admissiblecycles}
S(x_1,x_2) \cdot S(x_2,x_3) \cdots S(x_{n-1},x_n) \cdot S(x_n,x_1)=1,
\end{equation}
whenever all the factors are defined; this is what we prove in Section \ref{sectioncocyclecondition} and what is the main ingredient for the proof of our main theorem in Section \ref{sectiongauge}. To work properly with the domain where all factors of \eqref{eq:admissiblecycles} are defined, we introduce the \textit{equivalent kernel graph} in Section \ref{sectiongraph}. Now, we first state the basic properties of the transition function. For this, we note that by the symmetry of $K$ and $Q$, we have that $Z$ is symmetric, in the sense that $(x,y)\in Z$ if and only if $(y,x)\in Z$.

\begin{lemma}
\label{lem:propertiesS}
Suppose that $K$ and $Q$ are equivalent symmetric kernels and let $S$ be the transition function from $K$ to $Q$ as in \eqref{eq:defS}. Then we have that
\begin{align}
S(x,y)&=S(y,x), & \textrm{ for all } (x,y)\in X^2\setminus Z, \label{eq:symmetryS} \\
S(x,y)&\in \{\pm 1\},  & \textrm{ for all } (x,y)\in X^2\setminus Z, \label{eq:valuesofS} \\
S(x,x) &=1, & \textrm{ if } (x,x)\in X^2 \setminus Z. \label{eq:Sondiagonal}
\end{align}
\end{lemma}
\begin{proof}
The symmetry follows directly from the symmetry of $K$ and $Q$ and the definition \eqref{eq:defS} of the transition function. By \eqref{eq:secondfundamental} it follows that for all $(x,y)\in X^2 \setminus Z$, we have $S(x,y)^2=1$. Therefore, by Remark \ref{rem:pm1andowninverse}, $S(x,y)\in \{\pm 1\}$. The last statement of the lemma follows directly from \eqref{eq:firstfundamental}.
\end{proof}

We can immediately deduce the following if $\mathbb{F}$ has characteristic 2.

\begin{corollary}
\label{cor:char2}
If $\mathbb{F}$ has characteristic 2, and $K,Q: X^2 \rightarrow \mathbb{F}$ are equivalent symmetric kernels, then $K=Q$.
\end{corollary}
\begin{proof}
Since in $\mathbb{F}$ we have that $-1=1$, we have that $S(x,y)=1$ for all $(x,y)\in X^2 \setminus Z$ by Lemma \ref{lem:propertiesS}. From this, by \eqref{eq:defS} it immediately follows that $K=Q$.
\end{proof}

\section{The equivalent kernel graph}
\label{sectiongraph}

As mentioned above, we introduce a graph to deal with the fact that the transition function $S$ is only defined on points $(x,y)\in X^2 \setminus Z$.

\begin{definition}\label{def:kernelgraph}
The \textbf{equivalent kernel graph} $G$ is an undirected graph with the elements of $X$ as vertices and an edge between $x$ and $y$ if and only if $(x,y)\in X^2 \setminus Z$, i.e., if $K(x,y)\neq 0$.
\end{definition}

Note that in the definition of $G$, the requirement of the existence of an edge between $x$ and $y$ is symmetrical in $x$ and $y$, precisely since $K$ and $Q$ are symmetric kernels. Furthermore, in the situation that $Z=\emptyset$, the graph $G$ is a complete graph.

Since the notation and nomenclature that are used in graph theory varies from source to source, we explicitly state some of the notions as we use them. For this, let $G$ be any graph. A \textbf{path} from $x$ to $y$ in $G$ is a finite vector $p=\{p_i\}_{i=0}^n$ such that $p_0=x$ and $p_n=y$, and there is an edge between $p_i$ and $p_{i+1}$ in $G$. We call $n$ the \textbf{length} of the path $p$. A \textbf{simple path} is a path with distinct vertices, possibly except the starting and ending point. A \textbf{cycle} in a point $x\in X$ is a path from $x$ to itself. Then naturally, a simple cycle is a path that is both simple and a cycle.

If $p=\{p_i\}_{i=0}^n$ is a simple cycle, then $p$ is called \textbf{chordless} if every edge that connects two vertices in the cycle is already in the cycle itself. More explicitly, $p=\{p_i\}_{i=0}^n$ is chordless if  the existence of an edge between $p_i$ and $p_j$ implies that $i=j\pm 1$, $(i,j)=(0,n)$ or $(i,j)=(n,0)$.

\begin{remark}
The equivalent kernel graph $G$ is defined precisely in such a way that all factors in the left hand side of \eqref{eq:admissiblecycles} are defined if and only if $\{x_i\}_{i=1}^n$ is a cycle in $G$. Therefore, the cocycle condition \eqref{eq:cocyclecondition} is replaced by a requirement for all cycles in the graph $G$, see Proposition \ref{prop:loopcondition}.
\end{remark}

\section{The analogue of the cocycle condition}
\label{sectioncocyclecondition}

For any path $p=\{p_i\}_{i=0}^n$ in the graph $G$, we define
\begin{equation}\label{eq:defS[p]}
S[p]=\prod_{i=1}^n S(p_{i-1},p_i),
\end{equation}
and similarly we define $K[p]$ and $Q[p]$ for any path.
\begin{proposition}\label{prop:loopcondition}
For any cycle $p$ in the equivalent kernel graph $G$, we have $S[p]=1$.
\end{proposition}
\begin{proof}
Note that if $\mathbb{F}$ has characteristic 2, then this result is trivial by Corollary \ref{cor:char2}. For the rest of this proof we therefore assume that $\mathbb{F}$ does not have characteristic 2.

Now let $p=\{p_i\}_{i=1}^n$ be a cycle in the equivalent kernel graph $G$. We prove by induction on the length $n$ that $S[p]=1$. For $n=1$, we have $p_0=p_1$ and hence $S[p]=S(p_0,p_1)=1$ by \eqref{eq:Sondiagonal}. This establishes the induction basis.

Now suppose that $n\geq 2$ and we have proven that $S[q]=1$ for all cycles $q$ of length $k<n$. There are three different cases that cover all possibilities:
\begin{enumerate}
\item $p$ is not a simple cycle;
\item $p$ is a simple cycle, but not chordless;
\item $p$ is a chordless simple cycle.
\end{enumerate}
We show that in all three cases, we have that $S[p]=1$, which together establishes the induction step. We already note in advance that we only need the induction hypothesis for the first two cases.

\paragraph{The first case.}
If $p$ is not a simple cycle, then there are two integers $i$ and $j$ such that $0\leq i<j\leq n-1$ and $p_i=p_j$. Then define two new cycles:
\[q=(p_0,p_1,\dots,p_i,p_{j+1},p_{j+2},\dots,p_n),\]
\[q'=(p_i,p_{i+1},\dots,p_j).\]
We have that $q$ and $q'$ are both cycles since $p_i=p_j$, and the length of both $q$ and $q'$ are strictly smaller than $n$, so by the induction hypothesis $S[q]=S[q']=1$. Furthermore, by rearranging the factors in the defining equation \eqref{eq:defS[p]}, we see that
\[S[p]=S[q]\cdot S[q']=1,\]
which concludes the proof for this case.

\paragraph{The second case.}
If $p$ is a simple cycle that is not chordless, there are two integers $i$ and $j$ such that $0\leq i,j \leq n-1$, $i<j-1$, $(i,j)\neq (0,n-1)$ and $(p_i,p_j)$ is an edge in $G$. Again, we define two new cycles
\[r=(p_0,p_1,\dots,p_i,p_j,p_{j+1},\dots,p_n),\]
\[r'=(p_i,p_{i+1},\dots,p_j,p_i).\]
It is clear that the lengths of $r$ and $r'$ are both smaller than $n$ and hence  $S[r]=S[r']=1$ by the induction hypothesis. We note that $(p_i,p_j)$ is not an edge in $p$, but appears in both $r$ and $r'$, in opposite directions. Furthermore, the rest of the edges in $p$ appear precisely once in either $r$ or $r'$, but not in both. These make up all edges of $r$ and $r'$, whence by the definition \eqref{eq:defS[p]} we have
\[S[r]\cdot S[r'] =S[p]\cdot S(p_i,p_j) \cdot S(p_j,p_i).\]
by rearranging the factors. By Lemma \ref{lem:propertiesS} we know that
\[S(p_i,p_j) \cdot S(p_j,p_i) = S(p_i,p_j)^2=1,\]
and by combining all this we conclude that $S[p]=1$ too.

\paragraph{The third case.}
If $p$ is a chordless simple cycle, we prove that $K[p]=Q[p]$, which is equivalent to $S[p]=1$, by construction of $S$. For this, we do not need the induction hypothesis, but we extract the necessary information from the equation \eqref{eq:equivalentkernels} for the tuple $(p_1,\dots,p_n)$. Namely, we expand both sides of the equation as sums over permutations, such that we get 
\begin{equation}
\label{eq:expansionoverpermutations}
\sum_{\sigma \in S_n} \sgn(\sigma) \prod_{i=1}^n K(p_i,p_{\sigma(i)}) =\sum_{\sigma \in S_n} \sgn(\sigma) \prod_{i=1}^n Q(p_i,p_{\sigma(i)}).
\end{equation}
Since $p$ is a chordless simple cycle, the permutation $\sigma \in S_n$ only contributes to these sums if and only if $p_i$ and $p_{\sigma(i)}$ are either equal or neighbours in the cycle $p$, for all $1\leq i \leq n$; otherwise one of the factors is zero. This means that there are only two types of permutations $\sigma$ contributing to these sums:
\begin{enumerate}
\item $\sigma$ has only $1$- or $2$-cycles.
\item $\sigma=(1 \ 2 \ \cdots \ n)$ or $\sigma=(n \ n-1 \ \cdots \ 1)$. 
\end{enumerate}
If $\sigma$ is of the first type, it in fact only contributes if all 2-cycles connect two neighbours in $p$. Independent of this, if $\sigma$ is of the first type, we can apply \eqref{eq:firstfundamental} and \eqref{eq:secondfundamental} to all cycles of $\sigma$, such that we obtain
\[\prod_{i=1}^n K(p_i,p_{\sigma(i)}) = \prod_{i=1}^n Q(p_i,p_{\sigma(i)}).\]
If we subtract these contributions from \eqref{eq:expansionoverpermutations}, we only have the contributions of the permutations $\sigma=(1 \ 2 \ \cdots \ n)$ and $\sigma=(n \ n-1 \ \cdots \ 1)$ on both sides of the equation. In fact, by symmetry of $K$ and $Q$, these two contributions are the same, and we exactly arrive at $2K[p]=2Q[p]$. Since $\mathbb{F}$ does not have characteristic 2, we therefore have that $K[p]=Q[p]$, which concludes the proof.
\end{proof}

\section{Proof of Theorem \ref{thm:maintheorem}}
\label{sectiongauge}

Using Proposition \ref{prop:loopcondition}, we can now prove our main result. Note that in the proof the function $g_C$ is constructed similarly as the function $g$ in Remark \ref{rem:cocycle}.

\begin{proof}[Proof of Theorem \ref{thm:maintheorem}]
To prove the existence of the function $g$ such that \eqref{eq:gaugetransform} holds, we construct the function $g$ separately on the connected components of the equivalent kernel graph $G$. Namely, for every connected component $C$ of the equivalent kernel graph $G$, we construct a function $g_C: C\rightarrow \{\pm 1\}$ such that for every $x,y\in C$ that satisfy $(x,y)\in X^2 \setminus Z$, we have that $S(x,y)=g_C(x)g_C(y)$. Then, if we denote the connected component that contains a vertex $x$ by $[x]$, we can define $g: X \rightarrow  \{\pm 1\}$ by $g(x)=g_{[x]}(x)$. To see that this function $g$ satisfies \eqref{eq:gaugetransform}, we observe that if $(x,y)\in Z$, then $Q(x,y)=K(x,y)=0$, and hence \eqref{eq:gaugetransform} certainly holds for $(x,y)$. If we have $(x,y)\in X^2\setminus Z$, then $x$ and $y$ are in the same connected component (say $C$) of $G$ and hence
\[Q(x,y)=S(x,y) K(x,y)=g_C(x) g_C(y) K(x,y)= g(x) g(y) K(x,y),\]
whence \eqref{eq:gaugetransform} holds for all $(x,y)$. That proves Theorem \ref{thm:maintheorem}.

Therefore, we are left to consider a connected component $C$ of $G$ and  to construct the function $g_C: C \rightarrow \{\pm 1\}$ with the above properties. For this, fix a vertex $x_C$ in $C$. Now, take any vertex $y$ in $C$ and any path $p$ from $x_C$ to $y$. We claim that the value of $S[p]$ is independent of the path $p$, i.e., if $p'$ is another path from $x_C$ to $y$, then $S[p]=S[p']$. For this, suppose that $p$ has length $n$ and that $p'$ has the length $m$ and consider the path
\[q=(p_0,p_1,\dots,p_n,p'_{m-1},p'_{m-2},\dots,p'_0).\]
Indeed, this is a path precisely because $p_n=y=p'_m$ and in fact $q$ is a cycle in $x_C$. Therefore, we have
\[S[p]\cdot S[p'] = S[q]=1,\]
by Proposition \ref{prop:loopcondition}, so $S[p]=S[p']$, by Remark \ref{rem:pm1andowninverse} and \eqref{eq:valuesofS}. Then we can define
\[g_C(y)=S[p],\]
since this is independent of the choice of the path $p$ from $x_C$ to $y$. Now, suppose that $z,w\in C$ such that $(z,w)\in X^2 \setminus Z$. Then, take any path $r=\{r_i\}_{i=0}^l$ from $x_C$ to $z$. Such a path exists since $x_C$ and $z$ are in the same connected component of $G$. Next, define
\[r'=(r_0,r_1,\dots,r_l,w),\]
which is a path from $x_C$ to $w$. Then
\[g_C(w)=S[r']=S[r] \cdot S(z,w)=g_C(z)S(z,w),\]
so, since $g_C(z)$ is its own inverse, we have that
\[S(z,w)=g_C(z)g_C(w),\]
as desired. This concludes the proof.
\end{proof}

\begin{remark}
\label{rem:definiteuniqueness}
We can now give a complete answer to the question regarding the uniqueness of $g$, which was already addressed in Remark \ref{rem:uniquenessg}. Namely, if one takes a connected component $C$ and defines $\tilde{g}_C=-g_C$, then $\tilde{g}_C$ satisfies the same conditions as $g_C$. Hence, one can `change the sign' on all of the connected components separately. In fact, these are all the possible transformations $g \mapsto \tilde{g}$ that yield the same conjugation transformation. Namely, if one fixes the sign for the base point $x_C$ of a connected component, then that fixes the sign for all neighbouring points $y$ of $x_C$. By induction this fixes the sign on the whole connected component $C$. Hence, the possible choices of conjugation functions $g$ are labelled by the set of all functions $\operatorname{Conn}(G) \rightarrow \{\pm 1\}$, where $\operatorname{Conn}(G)$ denotes the set of connected components of $G$.
\end{remark}

\begin{appendices}
\end{appendices}

\section*{Acknowledgements}
I would like to thank Alexander Bufetov for introducing this interesting problem to me and for the discussions that followed. I am also grateful to Niels Bonneux and Arno Kuijlaars for carefully reading the manuscript. I am supported by EOS project 30889451 of the Flemish Science Foundation (FWO), the Belgian Interuniversity Attraction Pole P07/18 and partly by the long term structural funding-Methusalem grant of the Flemish Government.


\begin{thebibliography}{99}
\bibitem{Borodin} Borodin, A. (2011). Determinantal point processes. In G. Akemann, J. Baik and P. Di Francesco (Eds.), \textit{Oxford Handbook of
Random Matrix Theory} (pp. 231-249), Oxford: Oxford Univ. Press.

\bibitem{Bufetov_rigidity} Bufetov, A. I. (2016). Rigidity of determinantal point processes with the Airy, the Bessel and the Gamma kernel. \textit{Bulletin of Mathematical Sciences}, 6(1), 163-172.

\bibitem{Bufetov_personal} Bufetov, A. I. (2017). Personal communication.

\bibitem{Deift_Kriecherbauer_McLaughlin_Venakides_Zhou} Deift, P., Kriecherbauer, T., McLaughlin, K.T.R., Venakides, S. and Zhou, X. (1999). Strong asymptotics of orthogonal polynomials with respect to exponential weights. \textit{Communications on Pure and Applied Mathematics}, 52(12), 1491-1552.

\bibitem{Johansson} Johansson, K. (2006). Random Matrices and Determinantal Processes. In A. Bovier, et al. (Eds.), \textit{Mathematical Statistical Physics} (pp. 1-55). Amsterdam: Elsevier B.V.

\bibitem{Johansson_Edge} Johansson, K. (2018). Edge fluctuations of limit shapes. In Jerison D. et al. (Eds.), \textit{Current Developments in Mathematics 2016} (pp. 47-110), Somerville, MA: International Press.

\bibitem{Kuijlaars_2004} Kuijlaars, A. B. J., McLaughlin, K.T.R., Van Assche, W. and Vanlessen, M. (2004). The Riemann-Hilbert approach to strong asymptotics for orthogonal polynomials on $[-1, 1]$. \textit{Advances in Mathematics}, 188(2), 337-398.

\bibitem{Soshnikov} Soshnikov, A. (2000). Determinantal random point fields. \textit{Russian Mathematical Surveys}, 55(5), 923-975.

\bibitem{TracyWidom} Tracy, C.A and Widom, H. (1994). Level spacing distributions and the Bessel kernel. \textit{Communications in Mathematical Physics, 161(2)}, 289-309.

\end{thebibliography}
\end{document}